\documentclass[12pt]{amsart}
\usepackage{ifthen,verbatim}
\usepackage{floatflt,graphicx,graphics}

 \usepackage{tikz} 
\usepackage{mathrsfs}
\usepackage{amsmath,amssymb,amsthm}
\usepackage{mathtools}

\numberwithin{equation}{section}
\nonstopmode
\numberwithin{equation}{section}
\theoremstyle{plain}
\newtheorem{theorem}[equation]{Theorem}
\newtheorem{conjecture}[equation]{Conjecture}
\newtheorem{lemma}[equation]{Lemma}
\newtheorem{corollary}[equation]{Corollary}

\newtheorem{proposition}[equation]{Proposition}

\theoremstyle{definition}
\newtheorem{definition}[equation]{Definition}
\newtheorem{example}[equation]{Example}
\newtheorem{remark}[equation]{Remark}

\theoremstyle{remark}

\newcommand{\R}{\mathbb{R}}

\newcommand{\C}{\mathbb{C}}

\newcommand{\B}{\mathbb{B}}

\newcommand{\uhp}{\mathbb{H}}

{\qed\bigskip}

\newcounter{alphabet}

\newcounter{minutes}
\divide\time by 60
\newcounter{hours}
\multiply\time by 60
\addtocounter{minutes}{-\time}

\begin{document}
\bibliographystyle{amsplain}
\title[Intrinsic quasi-metrics]
{
Intrinsic quasi-metrics
}

\def\thefootnote{}
\footnotetext{
\texttt{\tiny File:~\jobname .tex,
          printed: \number\year-\number\month-\number\day,
          \thehours.\ifnum\theminutes<10{0}\fi\theminutes}
}
\makeatletter\def\thefootnote{\@arabic\c@footnote}\makeatother

\author[O. Rainio]{Oona Rainio}
\address{Department of Mathematics and Statistics, University of Turku, FI-20014 Turku, Finland}
\email{ormrai@utu.fi}

\begin{abstract}
The point pair function $p_G$ defined in a domain $G\subsetneq\mathbb{R}^n$ is shown to be a quasi-metric and its other properties are studied. For a convex domain $G\subsetneq\mathbb{R}^n$, a new intrinsic quasi-metric called the function $w_G$ is introduced. Several sharp results are established for these two quasi-metrics, and their connection to the triangular ratio metric is studied.
\end{abstract}
\maketitle

\section{Introduction}

In geometric function theory, one of the key concepts is an \emph{intrinsic} distance. This notion means a distance between two points fixed in a domain that not only depends on how close these points are to each other but also takes into account how they are located with respect to the boundary of the domain. A well-known example of an intrinsic metric is the \emph{hyperbolic metric} \cite{bm} but, especially during the past thirty years, numerous other \emph{hyperbolic type metrics} have been introduced, see \cite{chkv, fmv, hkvbook, h, imsz, ms, inm}.

This often raises the question about the reason for introducing new metrics and studying them instead of just focusing on those already existing. To answer this, it should be first noted that the slightly different definitions of the intrinsic metrics mean that they have unique advantages and suit for diverse purposes. Consequently, new metrics can be used to discover various intricate features of geometric entities that would not be detected with some other metrics. For instance, many hyperbolic type metrics behave slightly differently under quasiregular mappings and analysing these differences can give us a better understanding of how such mappings distort distances \cite{ps}.

Furthermore, new metrics can also bring information about the already existing metrics. Calculating the exact value of the hyperbolic metric in a domain that cannot be mapped onto the unit disk with a conformal mapping is often impossible but we can estimate it by using other intrinsic metrics with simpler definitions \cite[Ch. 4.3, pp. 59-66]{hkvbook}. However, in order to do this, we need to know the connection between the different metrics considered and to be able to create upper and lower bounds for them. Finding sharp inequalities for intrinsic metrics can often help us with some related applications and, for instance, in the estimation of condenser capacities \cite[Ch. 9, pp. 149-172]{hkvbook}.

Another noteworthy motivation for studying several different metrics is that their inequalities can tell us more about the domain where the metrics are defined. The definition for a uniform domain can be expressed with an inequality between the quasihyperbolic metric and the distance ratio metric as in \cite[Def. 6.1, p. 84]{hkvbook}. Similarly, some other inequalities can be used to determine whether the domain is, for instance, convex or not, like in Theorem \ref{sp_convex} below. Further, Corollary \ref{cor_halfspace} even shows an equality between metrics that serves as a condition for when the domain being a half-space.

In this paper, we consider two different intrinsic quasi-metrics. By a \emph{quasi-metric}, we mean a function that fulfills all the conditions of a metric otherwise but only a relaxed version of the triangle inequality instead of the inequality itself holds for this function, see Definition \ref{def_metric} and the inequality \eqref{quasi_inequality}. The first quasi-metric considered is the point pair function introduced by Chen et al. in 2015 \cite{chkv}, and the other quasi-metric is a function defined for the first time in Definition \ref{def_wcon} in this paper. We also study the triangular ratio metric introduced by P. H\"ast\"o in 2002 \cite{h} for one of the main results of this paper is showing how our new quasi-metric can be used to create a very good lower bound for this metric, especially in the case where the domain is the unit disk.

The structure of this paper is as follows. In Section \ref{sct3}, we study the properties of the point pair function and show how it can be used together with the triangular ratio metric to give us new information about the shape of the domain. Then, in Section \ref{sct4}, we introduce a new quasi-metric and show how it works as a lower bound for the triangular ratio metric in every convex domain. In Section \ref{sct5}, we focus on the case of the unit disk and find several sharp inequalities between different hyperbolic type metrics and quasi-metrics. Especially, we investigate how the new quasi-metric can be used to estimate the value of the triangular ratio metric in the unit disk, see Theorem \ref{thm_ws} and Conjecture \ref{con_ws}. 

{\bf Acknowledgements.} This research continues my work with Professor Matti Vuorinen in \cite{inm, sinb, sqm}. I am indebted to him for all guidance and other support. My research was also supported by Finnish Concordia Fund. Furthermore, I would like to thank the referees for their constructive suggestions and corrections.

\section{Preliminaries}

In this section, we will introduce the necessary definitions and some basic results related to them but let us first recall the definition of a metric.

\begin{definition}\label{def_metric}
For any non-empty space $G$, a \emph{metric} is a function $d:G\times G\to[0,\infty)$ that fulfills the following three conditions for all $x,y,z\in G$:\\
(1) Positivity: $d(x,y)\geq 0$, and $d(x,y)=0$ if and only if $x=y$,\\
(2) Symmetry: $d(x,y)=d(y,x)$,\\
(3) Triangle inequality: $d(x,y)\leq d(x,z)+d(z,y).$
\end{definition}

A \emph{quasi-metric} is a function $d$ that fulfills the definition above otherwise, but instead of the triangle inequality itself, it only fulfills the inequality
\begin{align}\label{quasi_inequality}
d(x,y)\leq c(d(x,z)+d(z,y))    
\end{align}
with some constant $c>1$ independent of the points $x,y,z$.

Now, let us introduce the notations used. Suppose that $G\subsetneq\R^n$ is some domain. For all $x\in G$, the Euclidean distance $d(x,\partial G)=\inf\{|x-z|\text{ }|\text{ }z\in\partial G\}$ will be denoted by $d_G(x)$. The Euclidean balls and spheres are written as $B^n(x,r)=\{y\in \R^n\text{ }|\text{ }|x-y|<r\}$, $\overline{B}^n(x,r)=\{y\in \R^n\text{ }|\text{ }|x-y|\leq r\}$ and $S^{n-1}(x,r)=\{y\in \R^n\text{ }|\text{ }|x-y|=r\}$. If $x$ or $r$ is not specified otherwise, suppose that $x=0$ and $r=1$. For points $x,y\in\R^n$, the Euclidean line passing through them is denoted by $L(x,y)$, the line segment between them by $[x,y]$ and the value of the smaller angle with vertex in the origin and $x,y$ on its sides by $\measuredangle XOY$. Furthermore, $\{e_1,...,e_n\}$ is the set of the unit vectors.

In this paper, we focus on the cases where the domain $G$ is either the upper half-space $\uhp^n=\{(x_1,...,x_n)\in\R^n\text{ }|\text{ }x_n>0\}$, the unit ball $\B^n=B^n(0,1)$ or the open sector $S_\theta=\{x\in\C\text{ }|\text{ }0<\arg(x)<\theta\}$ with an angle $\theta\in(0,2\pi)$. The hyperbolic metric can be defined in these cases with the formulas
\begin{align*}
\text{ch}\rho_{\uhp^n}(x,y)&=1+\frac{|x-y|^2}{2d_{\uhp^n}(x)d_{\uhp^n}(y)},\quad x,y\in\uhp^n,\\
\text{sh}^2\frac{\rho_{\B^n}(x,y)}{2}&=\frac{|x-y|^2}{(1-|x|^2)(1-|y|^2)},\quad x,y\in\B^n,\\
\rho_{S_\theta}(x,y)&=\rho_{\uhp^2}(x^{\pi\slash\theta},y^{\pi\slash\theta}),\quad x,y\in S_\theta,
\end{align*}
see \cite[(4.8), p. 52 \& (4.14), p. 55]{hkvbook}. In the two-dimensional unit disk, we can simply write
\begin{align*}
\text{th}\frac{\rho_{\B^2}(x,y)}{2}=\left|\frac{x-y}{1-x\overline{y}}\right|,
\end{align*}
where $\overline{y}$ is the complex conjugate of $y$.

For any domain $G\subsetneq\R^n$, define the following hyperbolic type metrics and quasi-metric:\newline
$(1)$ \cite[(1.1), p. 683]{chkv} The triangular ratio metric $s_G:G\times G\to[0,1],$ 
\begin{align*}
s_G(x,y)=\frac{|x-y|}{\inf_{z\in\partial G}(|x-z|+|z-y|)}, 
\end{align*}
$(2)$ \cite[2.2, p. 1123 \& Lemma 2.1, p. 1124]{hvz} the $j^*_G$-metric $j^*_G:G\times G\to[0,1],$
\begin{align*}
j^*_G(x,y)=\frac{|x-y|}{|x-y|+2\min\{d_G(x),d_G(y)\}},    
\end{align*}
$(3)$ \cite[p. 685]{chkv}, \cite[2.4, p. 1124]{hvz} the point pair function $p_G:G\times G\to[0,1],$
\begin{align*}
p_G(x,y)=\frac{|x-y|}{\sqrt{|x-y|^2+4d_G(x)d_G(y)}}.   
\end{align*}

\begin{remark}\label{rmk_invariantmetrics}
All three functions listed above are invariant under all similarity maps. In particular, when defined in a sector $S_\theta$, they are invariant under a reflection over the bisector of the sector and a stretching $x\mapsto r\cdot x$ with any $r>0$. Consequently, this allows us to make certain assumptions when choosing the points $x,y\in S_\theta$ to study these functions. 
\end{remark}

The metrics introduced above fulfill the following inequalities.

\begin{lemma}\label{lem_jsp_bounds}
\emph{\cite[Lemma 2.1, p. 1124; Lemma 2.2 \& Lemma 2.3, p. 1125 \& Thm 2.9(1), p. 1129]{hvz}} 
For any subdomain $G\subsetneq\R^n$ and all $x,y\in G$, the following inequalities hold:\newline
$(1)$ $j^*_G(x,y)\leq p_G(x,y)\leq\sqrt{2}j^*_G(x,y)$,\newline
$(2)$ $j^*_G(x,y)\leq s_G(x,y)\leq2j^*_G(x,y)$.\newline
Furthermore, if $G$ is convex, then for all $x,y\in G$\newline 
$(3)$ $s_G(x,y)\leq\sqrt{2}j^*_G(x,y)$.
\end{lemma}

\begin{lemma}\label{rhojsp_inG}
\emph{\cite[p. 460]{hkvbook}} For all $x,y\in G\in\{\uhp^n,\B^n\}$,
\begin{align*}
&(1)\quad{\rm th}\frac{\rho_{\uhp^n}(x,y)}{4}\leq j^*_{\uhp^n}(x,y)\leq s_{\uhp^n}(x,y)= p_{\uhp^n}(x,y)={\rm th}\frac{\rho_{\uhp^n}(x,y)}{2}\leq2{\rm th}\frac{\rho_{\uhp^n}(x,y)}{4},\\
&(2)\quad{\rm th}\frac{\rho_{\B^n}(x,y)}{4}\leq j^*_{\B^n}(x,y)\leq s_{\B^n}(x,y)\leq p_{\B^n}(x,y)\leq{\rm th}\frac{\rho_{\B^n}(x,y)}{2}\leq2{\rm th}\frac{\rho_{\B^n}(x,y)}{4}.
\end{align*}
\end{lemma}

Furthermore, the following results are often needed in order to calculate the value of the triangular ratio metric in the unit disk.

\begin{lemma}\label{lem_smetricinB_collinear}
\emph{\cite[11.2.1(1) p. 205]{hkvbook}}
For all $x,y\in\B^n$,
\begin{align*}
s_{\B^n}(x,y)\leq\frac{|x-y|}{2-|x+y|},    
\end{align*}
where the equality holds if the points $x,y$ are collinear with the origin.
\end{lemma}

\begin{theorem}\label{smetricinB_forconjugate}
\emph{\cite[Thm 3.1, p. 276]{hkvz}} If $x=h+ki\in\B^2$ with $h,k>0$, then
\begin{align*}
s_{\B^2}(x,\overline{x})&=|x|\text{ if }|x-\frac{1}{2}|>\frac{1}{2},\\
s_{\B^2}(x,\overline{x})&=\frac{k}{\sqrt{(1-h)^2+k^2}}\leq|x|\text{ otherwise.}
\end{align*}
\end{theorem}

\section{Point pair function}\label{sct3}

In this section, we will focus on the point pair function. The expression for this function was first introduced in \cite[p. 685]{chkv}, but it was named and researched further in \cite{hvz}. It was noted already in \cite[Rmk 3.1 p. 689]{chkv} that the point pair function defined in the unit disk is not a metric because it does not fulfill the triangle inequality for all points of this domain. However, the point pair function offers a good upper bound for the triangular ratio metric in convex domains \cite[Lemma 11.6(1), p. 197]{hkvbook} and, by Lemma \ref{rhojsp_inG}, it also serves as a lower bound for the expression ${\rm th}(\rho_G(x,y)\slash2)$ if $G\in\{\uhp^n,\B^n\}$ so studying its properties more carefully is relevant.

It is very easy to show that there is a constant $c>1$ with which the point pair function is a quasi-metric.

\begin{lemma}
For all domains $G\subsetneq\R^n$, the point pair function $p_G$ is a quasi-metric with a constant less than or equal to $\sqrt{2}$.
\end{lemma}
\begin{proof}
It follows from Lemma \ref{lem_jsp_bounds}(1) and the fact that the $j^*_G$-metric is a metric that
\begin{align*}
p_G(x,y)\leq \sqrt{2}j^*_G(x,y)\leq \sqrt{2}(j^*_G(x,z)+j^*_G(z,y))\leq\sqrt{2}(p_G(x,z)+p_G(z,y)).
\end{align*}
\end{proof}

For some special choices of the domain $G$, it is possible to find a better constant than $\sqrt{2}$ with which the point pair function is a quasi-metric. For instance, by Lemma \ref{rhojsp_inG}(1), $p_{\uhp^n}(x,y)=s_{\uhp^n}(x,y)$, so the point pair function is a metric in the domain $G=\uhp^n$ and trivially fulfills the inequality \eqref{quasi_inequality} with $c=1$. Furthermore, numerical tests suggest also that the constant $\sqrt{2}$ can be always replaced with a better one, even in the case where the exact shape of the domain is unknown.

\begin{conjecture}
For all domains $G\subsetneq\R^n$, the point pair function $p_G$ is a quasi-metric with a constant less than or equal to $\sqrt{5}\slash2$.
\end{conjecture}

However, for an arbitrary domain $G$, there cannot be a better constant than $\sqrt{5}\slash2$ with which the point pair function is a metric. Namely, if the domain $G$ is the unit ball $\B^n$, we see that we must choose $c\geq\sqrt{5}\slash2$ so that the inequality \eqref{quasi_inequality} holds for the points $x=e_1\slash3$, $z=0$ and $y=-e_1\slash3$. Other domains like this where the inequality \eqref{quasi_inequality} can only hold with $c\geq\sqrt{5}\slash2$ include, for instance, a twice punctured space $\R^n\backslash(\{s\}\cup\{t\})$, $s\neq t\in\R^n$, and all $k$-dimensional hypercubes in $\R^n$ where $1\leq k\leq n$. Consequently, the point pair function is not a metric in any of these domains.

It can be also shown that the point pair function $p_G$ is not a metric in a sector $S_\theta$ with an angle $0<\theta<\pi$. For instance, if $\theta=\pi\slash2$, then the points $x=e^{\pi i\slash5}$, $y=e^{3\pi i\slash10}$ and $z=(x+y)\slash2$ do not fulfill the triangle inequality. However, it is noteworthy that the point pair function is a metric in an open sector $S_\theta$ whenever the angle $\theta$ is from the interval $[\pi,2\pi)$. 

\begin{theorem}
In an open sector $S_\theta$ with an angle $\pi\leq\theta<2\pi$, the point pair function $p_{S_\theta}$ is a metric. 
\end{theorem}
\begin{proof}
Trivially, we only need to prove that the point pair function fulfills the triangle inequality in this domain. Fix distinct points $x,y\in S_\theta$. Note that if $\theta\geq\pi$ then, for every point $x\in S_\theta$, there is exactly one point $x'\in S^1(x,d_{S_\theta}(x))\cap\partial S_\theta$. Fix $x',y'$ like this for the points $x,y$, respectively. Furthermore, define $J$ as follows: If $x,y,x',y'$ are collinear, let $J=L(x,y)$; if $L(x,x')$ and $L(y,y')$ are two distinct parallel lines, let $J$ be the closed strip between them; and if $L(x,x')$ are $L(y,y')$ intersect at one point point, let $J$ be the smaller closed sector with points $x,x'$ on its one side and points $y,y'$ on its other side.

We are interested in such a point $z\in S_\theta$ that minimizes the sum $p_{S_\theta}(x,z)+p_{S_\theta}(z,y)$. Note that if $z\notin J$, then it can be rotated around either $x$ or $y$ into a new point $z\in J\cap S_\theta$ so that one of the distances $|x-z|$ and $|z-y|$ does not change and the other one decreases, and the distance $d_{S_\theta}(z)$ increases. Since these changes do not increase the sum $p_{S_\theta}(x,z)+p_{S_\theta}(z,y)$, we can suppose without loss of generality that the point $z$ must belong in $J\cap S_\theta$.

Choose now a half-plane $H$ such that $x,y\in H$ and $\partial H$ is a tangent for both $S^1(x,d_{S_\theta}(x))$ and $S^1(y,d_{S_\theta}(y))$. Now, $d_H(x)=d_{S_\theta}(x)$ and $d_H(y)=d_{S_\theta}(y)$. Clearly, $J\cap S_\theta\subset H$ and, since $\theta\geq\pi$, for every point $z\in J\cap S_\theta$, $d_{S_\theta}(z)\leq d_H(z)$. Recall that the point pair function $p_G$ is a metric in a half-plane domain. It follows that 
\begin{align*}
p_{S_\theta}(x,y)
&=p_H(x,y)
\leq\inf_{z\in J\cap S_\theta}(p_H(x,z)+p_H(z,y))
\leq\inf_{z\in J\cap S_\theta}(p_{S_\theta}(x,z)+p_{S_\theta}(z,y))\\
&=\inf_{z\in S_\theta}(p_{S_\theta}(x,z)+p_{S_\theta}(z,y)),
\end{align*}
which proves our result.
\end{proof}

Because the point pair function has several desirable properties that can be used when creating bounds for hyperbolic type metrics, it is useful to study it more carefully in sector domains. Next, we will show that the point pair function is invariant under a certain conformal mapping defined in a sector. Note that the triangular ratio metric has this same property, see \cite[Thm 4.16, p. 14]{inm}.

\begin{lemma}\label{f*_for_p_inS}
For all angles $0<\theta<2\pi$, the point pair function $p_{S_\theta}$ is invariant under the M\"obius transformation $f:S_\theta\to S_\theta$, $f(x)=x\slash|x|^2$.
\end{lemma}
\begin{proof}
By Remark \ref{rmk_invariantmetrics}, we can fix $x=e^{ki}$ and $y=re^{hi}$ with $r>0$ and $0<k\leq h<\theta$ without loss of generality. Now, $f(x)=x=e^{ki}$ and $f(y)=(1\slash r)e^{hi}$. It follows that
\begin{align*}
p_{S_\theta}(x,y)&=\frac{|1-re^{(h-k)i}|}{\sqrt{|1-re^{(h-k)i}|^2+4d_{S_\theta}(e^{ki})d_{S_\theta}(re^{hi})}}\\
&=\frac{r|1-(1\slash r)e^{(h-k)i}|}{\sqrt{r^2|1-(1\slash r)e^{(h-k)i}|^2+4r^2d_{S_\theta}(e^{ki})d_{S_\theta}((1\slash r)e^{hi})}}\\
&=\frac{|1-(1\slash r)e^{(h-k)i}|}{\sqrt{|1-(1\slash r)e^{(h-k)i}|^2+4d_{S_\theta}(e^{ki})d_{S_\theta}((1\slash r)e^{hi})}}
=p_{S_\theta}(f(x),f(y)),
\end{align*}
which proves the result.
\end{proof}

Let us yet consider the connection between the point pair function and the triangular ratio metric and, especially, what we can tell about the domain by studying these metrics.

\begin{theorem}\label{sp_convex}
\emph{\cite[Theorem 3.8, p. 5]{sqm}}
A domain $G\subsetneq\R^n$ is convex if and only if the inequality $s_G(x,y)\leq p_G(x,y)$ holds for all $x,y\in G$.
\end{theorem}

\begin{theorem}\label{thm_Gminusconvex}
If $G\subsetneq\R^n$ is a domain and the inequality $s_G(x,y)\geq p_G(x,y)$ holds for all $x,y\in G$, the complement $\R^n\backslash G$ is a connected, convex set. 
\end{theorem}
\begin{proof}
Suppose that $\R^n\backslash G$ is either non-convex or non-connected. Now, there are some $u,v\in\partial G$ such that $[u,v]\cap G\neq\varnothing$. It follows that there must be some ball $B^n(c,r)\subset G$ so that the intersection $S^n(c,r)\cap\partial G$ contains distinct points $u',v'$. Without loss of generality, we can assume that $c=0$ and $r=1$. Since $u',v'$ are distinct, $\mu=\measuredangle U'OV'\neq0$. If $\mu=\pi$, it holds for points $x=u'\slash2$ and $y=v'\slash2$ that
\begin{align*}
s_G(x,y)
=\frac{1}{2}
<\frac{1}{\sqrt{2}}
=\frac{1}{\sqrt{1^2+4(1\slash2)^2}}
=p_G(x,y).
\end{align*}
If $0<\mu<\pi$ instead, fix $x=\cos(\mu\slash2)u'$ and $y=\cos(\mu\slash2)v'$. By the law of cosines and the sine double-angle formula, we will have $|x-y|=\sin(\mu)$. It follows now from Theorem \ref{smetricinB_forconjugate} and a few trigonometric identities that
\begin{align*}
s_G(x,y)
&\leq s_{\B^n}(x,y)
\leq\cos(\mu\slash2)
=\frac{\sin(\mu)}{2\sin(\mu\slash2)}
=\frac{\sin(\mu)}{\sqrt{4\sin^2(\mu\slash2)}}\\
&=\frac{\sin(\mu)}{\sqrt{\sin^2(\mu)+4\sin^2(\mu\slash2)(1-\cos^2(\mu\slash2))}}
=\frac{\sin(\mu)}{\sqrt{\sin^2(\mu)+4(1-\cos^2(\mu\slash2))^2}}\\
&<\frac{\sin(\mu)}{\sqrt{\sin^2(\mu)+4(1-\cos(\mu\slash2))^2}}
=p_G(x,y).
\end{align*}
Consequently, if $\R^n\backslash G$ is not convex, there are always some points $x,y\in G$ such that the inequality $s_G(x,y)<p_G(x,y)$ holds and the theorem follows from this.
\end{proof}

\begin{corollary}\label{cor_halfspace}
If $G\subsetneq\R^n$ is a domain such that $s_G(x,y)=p_G(x,y)$ holds for all $x,y\in G$, then $G$ is a half-space. 
\end{corollary}
\begin{proof}
By Theorems \ref{sp_convex} and \ref{thm_Gminusconvex}, both the sets $G$ and $\R^n\backslash G$ must be convex, from which the result follows directly.  
\end{proof}

\section{New quasi-metric}\label{sct4}

In this section, we define a new intrinsic quasi-metric $w_G$ in a convex domain $G$ and study its basic properties. As can be seen from Theorem \ref{thm_swconvex}, this function gives a lower bound for the triangular ratio metric. Since the point pair function serves as an upper bound for the triangular ratio metric, these two quasi-metrics can be used to form bounds for the triangular ratio distance like in Corollary \ref{cor_jwsp_convex} below. Furthermore, these three functions are equivalent in the case of the half-space, see Proposition \ref{prop_wsp_inH}, so these bounds are clearly essentially sharp at least in some cases.

First, consider the following definition.

\begin{definition}\label{def_wcon}
Let $G\subsetneq\R^n$ be a convex domain. For any $x\in G$, there is a non-empty set
\begin{align*}
\widetilde{X}=\{\widetilde{x}\in S^{n-1}(x,2d_G(x))\text{ }|\text{ }(x+\widetilde{x})\slash2\in\partial G\}.    
\end{align*}
Define now a function $w_G:G\times G\to[0,1]$,
\begin{align*}
w_G(x,y)&=\frac{|x-y|}{\min\{\inf_{\widetilde{y}\in\widetilde{Y}}|x-\widetilde{y}|,\inf_{\widetilde{x}\in\widetilde{X}}|y-\widetilde{x}|\}},
\quad x,y\in G.
\end{align*}
\end{definition}

Note that we can only define the function $w_G$ for convex domains $G$ because, for a non-convex domain $G$ and some points $x,y\in G$, there are some points $x,y\in G$ such that $y=\widetilde{x}$ with some $\widetilde{x}\in\widetilde{X}$ and the denominator in the expression of $w_G$ would become zero.

\begin{proposition}\label{prop_wsp_inH}
For all points $x,y\in\uhp^n$, 
\begin{align*}
w_{\uhp^n}(x,y)=s_{\uhp^n}(x,y)=p_{\uhp^n}(x,y).    
\end{align*}
\end{proposition}
\begin{proof}
For all $x=(x_1,...,x_n)\in\uhp^n$, there is only one point $\widetilde{x}=(x_1,...,x_{n-1},-x_n)=x-2x_ne_n$ in the set $\widetilde{X}$. Thus, for all $x,y\in\uhp^n$, 
\begin{align*}
\frac{w_{\uhp^n}(x,y)}{p_{\uhp^n}(x,y)}=\frac{\sqrt{|x-y|^2+4x_ny_n}}{\min\{|x-\widetilde{y}|,|y-\widetilde{x}|\}}
=\frac{\sqrt{|x-y|^2+4x_ny_n}}{\min\{|x-y+2y_ne_n|,|y-x+2x_ne_n|\}}=1.   
\end{align*}
The result $s_{\uhp^n}(x,y)=p_{\uhp^n}(x,y)$ is in Lemma \ref{rhojsp_inG}(1). 
\end{proof}

\begin{lemma}\label{lem_winS}
For all points $x,y\in S_\theta$ with $0<\theta\leq\pi$, $s_{S_\theta}(x,y)=w_{S_\theta}(x,y)$.
\end{lemma}
\begin{proof}
By the known solution to Heron's problem, the triangular ratio distance between $x,y\in S_\theta$ is
\begin{align}\label{swS_dist}
s_{S_\theta}(x,y)=\frac{|x-y|}{\inf_{z\in\partial S_\theta}(|x-z|+|z-y|)}=\frac{|x-y|}{\min\{|x'-y|,|x-y'|\}},  \end{align}
where the points $x',y'$ are the points $x,y$ reflected over the closest side of the sector $S_\theta$, respectively. Since the points in sets $\widetilde{X}$ and $\widetilde{Y}$ are similarly found by reflecting $x,y$ over the closest sides, the distance $w_{S_\theta}(x,y)$ is equivalent to \eqref{swS_dist} and the result follows.
\end{proof}

While it trivially follows from Lemma \ref{lem_winS} that the function $w_G$ is a metric in the case $G=S_\theta$ with some $0<\theta\leq\pi$, this is not true for all convex domains $G$, as can be seen with the following example.

\begin{example}\label{ex_forw}
$G=\{z\in\C\text{ }|\text{ }-1<\text{Re}(z)<1,0<\text{Im}(z)<1\}$ is a convex domain, in which $w_G$ is not a metric.
\end{example}
\begin{proof}
If $x=1\slash2+k+i\slash2$, $y=-1\slash2+i\slash2$ and $z=-1\slash2-k+i\slash2$ with $0<k<1\slash3$, it follows that
\begin{align*}
w_G(x,y)=\frac{1+k}{\sqrt{1+(1+k)^2}},\quad
w_G(x,z)=\frac{1+2k}{2},\quad
w_G(z,y)=\frac{k}{1-k}   
\end{align*}
and, consequently,
\begin{align*}
\lim_{k\to0^+}\frac{w_G(x,y)}{w_G(x,z)+w_G(z,y)}=\lim_{k\to0^+}\frac{2(1-k^2)}{\sqrt{1+(1+k)^2}(1+3k-2k^2)}=\sqrt{2}. 
\end{align*}
\end{proof}

Let us next show that $w_G$ is a quasi-metric by finding first the inequalities between it and two hyperbolic type metrics, the $j^*_G$-metric and the triangular ratio metric, in convex domains.

\begin{proposition}\label{prop_jw_con}
For any convex domain $G\subsetneq\R^n$ and all $x,y\in G$, $j^*_G(x,y)\leq w_G(x,y)$.
\end{proposition}
\begin{proof}
By the triangle inequality,
\begin{align*}
&\min\{\inf_{\widetilde{y}\in\widetilde{Y}}|x-\widetilde{y}|,\inf_{\widetilde{x}\in\widetilde{X}}|y-\widetilde{x}|\}
\leq\min\{|x-y|+d(y,\widetilde{Y}),|x-y|+d(x,\widetilde{X})\}\\
&=\min\{|x-y|+2d_G(y),|x-y|+2d_G(x)\}
=|x-y|+2\min\{d_G(x),d_G(y)\},
\end{align*}
so the result follows.
\end{proof}

\begin{theorem}\label{thm_swconvex}
For an arbitrary convex domain $G\subsetneq\R^n$ and all $x,y\in G$,
\begin{align*}
w_G(x,y)\leq s_G(x,y)\leq\sqrt{2}w_G(x,y).    
\end{align*}
\end{theorem}
\begin{proof}
Choose any distinct $x,y\in G$. By symmetry, we can suppose that $\inf_{\widetilde{x}\in\widetilde{X}}|y-\widetilde{x}|\leq\inf_{\widetilde{y}\in\widetilde{Y}}|x-\widetilde{y}|$. Fix $\widetilde{x}$ as the point that gives this smaller infimum. Let us only consider the two-dimensional plane containing $x,y,\widetilde{x}$ and set $n=2$. Fix $u=[x,\widetilde{x}]\cap\partial G$ and $z=[y,\widetilde{x}]\cap\partial G$. Position the domain $G$ on the upper half-plane so that the real axis is the tangent of $S^1(x,d_G(x))$ at the point $u$. Since $G$ is convex, it must be a subset of $\uhp^2$ and therefore $z\in\uhp^2\cup\R$. Thus, it follows that $|x-z|\leq|z-\widetilde{x}|$. Consequently,
\begin{align*}
w_G(x,y)=\frac{|x-y|}{|y-\widetilde{x}|}
=\frac{|x-y|}{|z-\widetilde{x}|+|z-y|}
\leq\frac{|x-y|}{|x-z|+|z-y|}
\leq s_G(x,y).
\end{align*}
The inequality $s_G(x,y)\leq\sqrt{2}w_G(x,y)$ follows from Lemma \ref{lem_jsp_bounds}(3) and Proposition \ref{prop_jw_con}.
\end{proof}

Now, we can show that the function $w_G$ is a quasi-metric.

\begin{corollary}\label{cor_wquasi}
For an arbitrary convex domain $G\subsetneq\R^n$, the function $w_G$ is a quasi-metric with a constant less than or equal to $\sqrt{2}$, and the number $\sqrt{2}$ here is sharp.
\end{corollary}
\begin{proof}
It follows from Theorem \ref{thm_swconvex} and the fact that the triangular ratio metric is always a metric that
\begin{align*}
w_G(x,y)\leq s_G(x,y)\leq s_G(x,z)+s_G(z,y)\leq\sqrt{2}(w_G(x,z)+w_G(z,y))   
\end{align*}
 and, by Example \ref{ex_forw}, the constant $\sqrt{2}$ here is the best one possible for an arbitrary convex domain.
\end{proof}

We will also have the following result.

\begin{corollary}\label{cor_jwsp_convex}
For any convex domain $G\subsetneq\R^n$ and all $x,y\in G$,
\begin{align*}
j^*_G(x,y)\leq w_G(x,y)\leq s_G(x,y)\leq p_G(x,y).    
\end{align*}
\end{corollary}
\begin{proof}
Follows from Proposition \ref{prop_jw_con} and Theorems \ref{thm_swconvex} and \ref{sp_convex}.
\end{proof}

Next, in order to summarize our results found above, let us yet write Lemma \ref{rhojsp_inG} with the quasi-metric $w_G$.

\begin{corollary}\label{cor_wrho}
\emph{\cite[p. 460]{hkvbook}} For all $x,y\in G\in\{\uhp^n,\B^n\}$,
\begin{align*}
&(1)\quad{\rm th}\frac{\rho_{\uhp^n}(x,y)}{4}\leq j^*_{\uhp^n}(x,y)\leq w_{\uhp^n}(x,y)=s_{\uhp^n}(x,y)= p_{\uhp^n}(x,y)={\rm th}\frac{\rho_{\uhp^n}(x,y)}{2},\\
&(2)\quad{\rm th}\frac{\rho_{\B^n}(x,y)}{4}\leq j^*_{\B^n}(x,y)\leq w_{\B^n}(x,y)\leq s_{\B^n}(x,y)\leq p_{\B^n}(x,y)\leq{\rm th}\frac{\rho_{\B^n}(x,y)}{2}. 
\end{align*}
\end{corollary}
\begin{proof}
Follows from Lemma \ref{rhojsp_inG}, Proposition \ref{prop_wsp_inH} and Corollary \ref{cor_jwsp_convex}.
\end{proof}

\section{Quasi-metrics in the unit disk}\label{sct5}

In this section, we will focus on the inequalities between the hyperbolic type metrics and quasi-metrics in the case of the unit disk. Calculating the exact value of the triangular ratio metric in the unit disk is not a trivial task, but instead quite a difficult problem with a very long history, see \cite{fhmv} for more details. However, we already know from Corollary \ref{cor_jwsp_convex} that the quasi-metric $w_G$ serves as a lower bound for the triangular ratio metric in convex domains $G$ and this helps us considerably.

\begin{remark}\label{rmk_b2Intobn}
Note that while we focus below mostly on the unit disk $\B^2$, all the inequalities can be extended to the general case with the unit ball $\B^n$, because the values of the metrics and quasi-metrics considered only depend on how the points $x,y$ are located on the two-dimensional place containing them and the origin.
\end{remark}

First, we will define the function $w_G$ of Definition \ref{def_wcon} in the case $G=\B^n$. Denote below $\widetilde{x}=x(2-|x|)\slash|x|$ for all points $x\in\B^n\backslash\{0\}$. We will have the following results.

\begin{proposition}\label{prop_tildedist}
If $x,y\in\B^n\backslash\{0\}$ such that $|y|\leq|x|$, then $|y-\widetilde{x}|\leq|x-\widetilde{y}|$.
\end{proposition}
\begin{proof}
Let $\mu=\measuredangle XOY$. Note that $|\widetilde{x}|=2-|x|$ and $|\widetilde{y}|=2-|y|$. By the law of cosines,
\begin{align*}
&|y-\widetilde{x}|\leq|x-\widetilde{y}|\\
\Leftrightarrow\quad
&\sqrt{|y|^2+(2-|x|)^2-2|y|(2-|x|)\cos(\mu)}\leq\sqrt{|x|^2+(2-|y|)^2-2|x|(2-|y|)\cos(\mu)}\\
\Leftrightarrow\quad
&|y|^2-(2-|y|)^2-|x|^2+(2-|x|)^2+2|x|(2-|y|)\cos(\mu)-2|y|(2-|x|)\cos(\mu)\leq0\\
\Leftrightarrow\quad
&4|y|+4-4|x|-4+2(|x|(2-|y|)-|y|(2-|x|))\cos(\mu)\leq0\\
\Leftrightarrow\quad
&4(|y|-|x|)+4(|x|-|y|)\cos(\mu)\leq0\quad
\Leftrightarrow\quad
(|y|-|x|)(1-\cos(\mu))\leq0\\
\Leftrightarrow\quad
&|y|\leq|x|,
\end{align*}
which proves the result.
\end{proof}

\begin{proposition}\label{prop_contiw}
If $x\in\B^n\backslash\{0\}$ is fixed and $y\to0$, then
\begin{align*}
\frac{|x-y|}{|y-\widetilde{x}|}\to
\frac{|x|}{2-|x|}.    
\end{align*}
\end{proposition}
\begin{proof}
By writing $\mu=\measuredangle XOY$ and using the law of cosines,
\begin{align*}
\lim_{|y|\to0^+}\frac{|x-y|}{|y-\widetilde{x}|}
=\lim_{|y|\to0^+}\sqrt{\frac{|x|^2+|y|^2-2|x||y|\cos(\mu)}{|y|^2+(2-|x|)^2-2|y|(2-|x|)\cos(\mu)}}
=\sqrt{\frac{|x|^2}{(2-|x|)^2}}=\frac{|x|}{2-|x|}.
\end{align*}
\end{proof}

Now, consider the following result. 

\begin{proposition}\label{prop_winB}
In the domain $G=\B^n$, the quasi-metric $w_G$ is a function $w_{\B^n}:\B^n\times\B^n\to[0,1]$,
\begin{align*}
w_{\B^n}(x,y)&=\frac{|x-y|}{\min\{|x-\widetilde{y}|,|y-\widetilde{x}|\}},
\quad x,y\in\B^n\backslash\{0\},\\
w_{\B^n}(x,0)&=\frac{|x|}{2-|x|},\quad x\in\B^n,
\end{align*}
where $\widetilde{x}=x(2-|x|)\slash|x|$ and $\widetilde{y}=y(2-|y|)\slash|y|$.
\end{proposition}
\begin{proof}
Follows from \ref{def_wcon}.
\end{proof}

Consider also the next corollary, which follows directly from the proposition above and our earlier observations.

\begin{corollary}\label{cor_winBdef}
For all distinct points $x,y\in\B^n$ such that $0\leq|y|\leq|x|<1$,
\begin{align*}
w_{\B^n}(x,y)=\frac{|x-y|}{|y-\widetilde{x}|}.
\end{align*}
\end{corollary}
\begin{proof}
Note that $0\leq|y|\leq|x|$ and $x\neq y$, so $x\neq 0$. If $y=0$, the result follows from Proposition \ref{prop_winB} because $|0-\widetilde{x}|=2-|x|$. If $x,y\in\B^2\backslash\{0\}$ instead, the result holds by Propositions \ref{prop_winB} and \ref{prop_tildedist}.
\end{proof}

It follows from this and Proposition \ref{prop_contiw} that the function $w_{\B^n}$ defined as in Proposition \ref{prop_winB} is continuous. By Corollary \ref{cor_wquasi}, the function $w_{\B^n}$ is also at least a quasi-metric. In fact, according to the numerical tests, the function $w_{\B^n}$ seems to fulfill the triangle inequality in the unit ball, which would mean that the following conjecture holds.  

\begin{conjecture}
The function $w_{\B^n}$ is a metric on the unit ball.
\end{conjecture}

However, it does not affect the results of this paper if the function $w_{\B^n}$ truly is a metric or just a quasi-metric, so let us move on and show that the function $w_{\B^2}$ is quite a good lower bound for the triangular ratio metric in the unit disk.

\begin{theorem}\label{thm_ws}
For all $x,y\in\B^2$, $w_{\B^2}(x,y)\leq s_{\B^2}(x,y)$ and the equality holds here whenever $x,y$ are collinear with the origin.
\end{theorem}
\begin{proof}
The inequality follows from Corollary \ref{cor_wrho}(2). If the points $x,y\in\B^2$ are collinear with the origin, we can fix $x,y\in\R$ so that $0<-x<y\leq x<1$ without loss of generality. In the special case $x=y$, the equality holds trivially and, if $x\neq y$, by Lemma \ref{lem_smetricinB_collinear} and Corollary \ref{cor_winBdef},
\begin{align*}
s_{\B^2}(x,y)
=\frac{|x-y|}{2-|x+y|}
=\frac{x-y}{2-(x+y)}
=\frac{|x-y|}{|y-(2-x)|}
=\frac{|x-y|}{|y-\widetilde{x}|}
=w_{\B^2}(x,y).
\end{align*}
\end{proof}

By \cite[Lemma 3.12, p. 7]{sinb}, the following function is a lower bound for the triangular ratio metric.

\begin{definition}\cite[Def. 3.9, p. 7]{sinb}
For $x,y\in\B^2\backslash\{0\}$, define
\begin{align*}
\text{low}(x,y)=\frac{|x-y|}{\min\{|x-y^*|,|x^*-y|\}},  \end{align*}
where $x^*=x\slash|x|^2$ and $y^*=y\slash|y|^2$.
\end{definition}

However the quasi-metric $w_{\B^2}$ is a better lower bound for the triangular ratio metric in the unit disk than this low-function, as we will show below. 

\begin{lemma}
For all distinct points $x,y\in\B^2\backslash\{0\}$, $w_{\B^2}(x,y)>{\rm low}(x,y)$.
\end{lemma}
\begin{proof}
Let $x\in\B^2$, $k>1$ and $\mu$ be the value of the larger angle between line $L(1,x)$ and the real axis. Now,
\begin{align*}
|x-ke_1|=\sqrt{|x-1|^2+(k-1)^2-2|x-1|(k-1)\cos(\mu)}.    
\end{align*}
Here, $\cos(\mu)<0$ since $\pi\slash2<\mu\leq\pi$. Thus, we see that the distance $|x-ke_1|$ is strictly increasing with respect to $k$. In other words, the further away a point $k\in\R^2\backslash\overline{\B}^2$ is from the origin, the longer the distance between $k$ and an arbitrary point $x\in\B^2$ is. For every point $y\in\B^2\backslash\{0\}$,
\begin{align*}
1<2-|y|<\frac{1}{|y|}\quad\Leftrightarrow\quad
1<\left|\frac{y(2-|y|)}{|y|}\right|<\left|\frac{y}{|y|^2}\right|
\quad\Leftrightarrow\quad
1<|\widetilde{y}|<|y^*|,    
\end{align*}
so it follows by the observation made above that, for all $x\in\B^2$,
\begin{align*}
|x-\widetilde{y}|<|x-y^*|.    
\end{align*}
Consequently, by symmetry, the inequality
\begin{align*}
w_{\B^2}(x,y)=\frac{|x-y|}{\min\{|x-\widetilde{y}|,|y-\widetilde{x}|\}}>\frac{|x-y|}{\min\{|x-y^*|,|y-x^*|\}}
=\text{low}(x,y)  
\end{align*}
holds for all distinct points $x,y\in\B^2\backslash\{0\}$.
\end{proof}

Next, we will prove one sharp inequality between the two quasi-metrics considered this paper. 

\begin{theorem}\label{thm_wp}
For all points $x,y\in\B^2$,
\begin{align*}
w_{\B^2}(x,y)\leq p_{\B^2}(x,y)\leq\sqrt{2}w_{\B^2}(x,y),    
\end{align*}
where the equality $w_{\B^2}(x,y)=p_{\B^2}(x,y)$ holds whenever $x,y$ are on the same ray starting from the origin, and $p_{\B^2}(x,y)=\sqrt{2}w_{\B^2}(x,y)$ holds when $x=-y$ and $|x|=|y|=1\slash2$.
\end{theorem}
\begin{proof}
The first inequality follows from Corollary \ref{cor_wrho}(2). If $x=y$, the equality $w_{\B^2}(x,y)=p_{\B^2}(x,y)=0$ is trivial. Thus, if $x,y\in\B^2$ are on the same ray, we can now fix $0\leq y<x<1$ without loss of generality. Now, by Corollary \ref{cor_winBdef},
\begin{align*}
w_{\B^2}(x,y)=\frac{x-y}{2-x-y}=\frac{x-y}{\sqrt{(x-y)^2+4(1-x)(1-y)}}=p_{\B^2}(x,y).
\end{align*}

Next, let us prove the latter part of the inequality. We need to find that the maximum of the quotient
\begin{align}\label{quo_pw}
\frac{p_{\B^2}(x,y)}{w_{\B^2}(x,y)}=\frac{\min\{|x-\widetilde{y}|,|y-\widetilde{x}|\}}{\sqrt{|x-y|^2+4(1-|x|)(1-|y|)}}.    
\end{align}
In order to do that, we can suppose without loss of generality that $x,y$ are on different rays starting from the origin since, as we proved above, the equality $w_{\B^2}(x,y)=p_{\B^2}(x,y)$ holds otherwise. Choose these points so that $0<|y|\leq|x|<1$ and $\mu=\measuredangle XOY>0$. It follows from Corollary \ref{cor_winBdef} that the quotient \eqref{quo_pw} is now
\begin{align}
\frac{p_{\B^2}(x,y)}{w_{\B^2}(x,y)}
&=\frac{|y-\widetilde{x}|}{\sqrt{|x-y|^2+4(1-|x|)(1-|y|)}}\nonumber\\
&=\sqrt{\frac{|y|^2+(2-|x|)^2-2|y|(2-|x|)\cos(\mu)}{|x|^2+|y|^2-2|x||y|\cos(\mu)+4-4|x|-4|y|+4|x||y|}}\nonumber\\
&=\sqrt{\frac{|y|^2+(2-|x|)^2-2|y|(2-|x|)\cos(\mu)}{|y|^2+(2-|x|)^2-4|y|(1-|x|)-2|x||y|\cos(\mu)}}.\label{quo_squareroot_pw}
\end{align}
Fix now
\begin{align*}
j&=\cos(\mu),\quad
s=|y|^2+(2-|x|)^2,\quad  
t=2|y|(2-|x|),\\
u&=|y|^2+(2-|x|)^2-4|y|(1-|x|),\quad
v=2|x||y|,  
\end{align*}
so that the argument of the square root in the expression \eqref{quo_squareroot_pw} can be described with a function $f:[0,1]\to\R$,
\begin{align*}
f(j)=\frac{s-tj}{u-vj}.    
\end{align*}
By differentiation, the function $f$ is decreasing with respect to $j$, if and only if
\begin{align*}
f'(j)=\frac{-t(u-vj)+v(s-tj)}{(u-vj)^2}=\frac{sv-tu}{(u-vj)^2}\leq0
\quad\Leftrightarrow\quad
sv-tu\leq0.
\end{align*}
Since this last inequality is equivalent to
\begin{align*}
&(|y|^2+(2-|x|)^2)2|x||y|-2|y|(2-|x|)(|y|^2+(2-|x|)^2-4|y|(1-|x|))\leq0\\
\Leftrightarrow\quad&(|y|^2+(2-|x|)^2)|x|-(2-|x|)(|y|^2+(2-|x|)^2-4|y|(1-|x|))\leq0\\
\Leftrightarrow\quad&|y|^2|x|+|x|(2-|x|)^2-|y|^2(2-|x|)-(2-|x|)^3+4|y|(1-|x|)(2-|x|)\leq0\\
\Leftrightarrow\quad&-2|y|^2(1-|x|)-2(2-|x|)^2(1-|x|)+4|y|(1-|x|)(2-|x|)\leq0\\
\Leftrightarrow\quad&|y|^2+(2-|x|)^2-2|y|(2-|x|)\geq0\\
\Leftrightarrow\quad&(|y|-(2-|x|))^2=(2-|x|-|y|)^2\geq0,
\end{align*}
which clearly holds, it follows that the function $f$ and the quotient \eqref{quo_pw} are decreasing with respect to $j=\cos(\mu)$. The minimum value of $\cos(\mu)$ is $-1$ at $\mu=\pi$. Consequently, we can fix the points $x,y$ so that $x=h$ and $y=-h+k$ with $0\leq k<h<1$, without loss of generality. Now, the quotient \eqref{quo_pw} is
\begin{align*}
\frac{p_{\B^2}(x,y)}{w_{\B^2}(x,y)}
&=\frac{2-k}{\sqrt{(2h+k)^2+4(1-h)(1-h+k)}}
=\sqrt{\frac{4-4k+k^2}{8h^2-8h+k^2+4k+4}}\\
&\leq\sqrt{\frac{4-4k+k^2+k(4-k)}{8h^2-8h+k^2+4k+4-k(4+k)}}
=\sqrt{\frac{4}{8h^2-8h+4}}
=\frac{1}{\sqrt{2h^2-2h+1}}.
\end{align*}
This upper bound found above is the value of the quotient \eqref{quo_pw} in the case $k=0$, because, for $x=h$ and $y=-h$,
\begin{align*}
\frac{p_{\B^2}(x,y)}{w_{\B^2}(x,y)}
=\frac{2}{\sqrt{4h^2+4(1-h)^2}}
=\frac{1}{\sqrt{2h^2-2h+1}}.
\end{align*}
The expression $2h^2-2h+1$ obtains its minimum value $1\slash2$ at $h=1\slash2$, so the maximum value of the quotient \eqref{quo_pw} is $\sqrt{2}$.
\end{proof}

The next result follows.

\begin{theorem}\label{thm_wjinB}
For all $x,y\in\B^2$,
\begin{align*}
j^*_{\B^2}(x,y)\leq w_{\B^2}(x,y)\leq\sqrt{2}j^*_{\B^2}(x,y),    
\end{align*}
where the equality $j^*_{\B^2}(x,y)=w_{\B^2}(x,y)$ holds whenever $x,y$ are on the same ray starting from the origin and the constant $\sqrt{2}$ is the best possible one.
\end{theorem}
\begin{proof}
The inequality $j^*_{\B^2}(x,y)\leq w_{\B^2}(x,y)$ follows from Corollary \ref{cor_wrho}(2) and the inequality $w_{\B^2}(x,y)\leq\sqrt{2}j^*_{\B^2}(x,y)$ from Lemma \ref{lem_jsp_bounds}(1) and Theorem \ref{thm_wp}. If $x,y\in\B^2$ are on the same ray, we can suppose without loss of generality that $0\leq y\leq x<1$. If now $x=y$, the equality $j^*_{\B^2}(x,y)=w_{\B^2}(x,y)=0$ is trivial, and if $x\neq y$ instead, by Corollary \ref{cor_winBdef},
\begin{align*}
w_{\B^2}(x,y)=\frac{x-y}{2-x-y}=\frac{x-y}{x-y+2(1-x)}=j^*_{\B^2}(x,y).    
\end{align*}
Fix yet $x=1-k$ and $y=(1-k)e^{2ki}$ with $0<k<1$. By the law of cosines and the cosine double-angle formula,
\begin{align}
\frac{w_{\B^2}(x,y)}{j^*_{\B^2}(x,y)}
&=\frac{(1-k)|1-e^{2ki}|+2k}{|(1-k)e^{2ki}-(1+k)|}
=\frac{2(1-k)\sin(k)+2k}{\sqrt{2+2k^2-2(1-k^2)\cos(2k)}}\nonumber\\
&=\frac{(1-k)\sin(k)+k}{\sqrt{k^2+(1-k^2)\sin^2(k)}}.\label{quo_jw}
\end{align}
Since the quotient \eqref{quo_jw} has a limit value of $\sqrt{2}$ when $k\to0^+$, the final part of the theorem follows.
\end{proof}

Let us next focus on how the quasi-metric $w_{\B^2}$ can be used to create an upper bound for the triangular ratio metric. We know from Theorem \ref{thm_swconvex} that in the general case where the  domain $G$ is convex, the inequality $s_G(x,y)\leq\sqrt{2}w_G(x,y)$ holds. Thus, this must also hold in the unit disk, but several numerical tests suggest that the constant $\sqrt{2}$ is not necessarily the best possible when $G=\B^2$. The next result tells the best constant in a certain special case.

\begin{lemma}\label{lem_sw_specialcase}
For all $x,y\in\B^2$ such that $|x|=|y|$ and $\measuredangle XOY=\pi\slash2$,
\begin{align*}
s_{\B^2}(x,y)\leq c\cdot w_{\B^2}(x,y)\quad\text{with}\quad 
c=\sqrt{\frac{h_0^2-2h_0+2}{2h_0^2-2\sqrt{2}h_0+2}},\quad
h_0=\frac{1-\sqrt{9-6\sqrt{2}}}{2-\sqrt{2}}.
\end{align*}
\end{lemma}
\begin{proof}
Let $x=h$ and $y=hi$ for $0<h<1$. Because
\begin{align*}
|h-\frac{1}{2}e^{\pi i\slash4}|>\frac{1}{2}
\quad\Leftrightarrow\quad
|2\sqrt{2}h-1-i|>\sqrt{2}
\quad\Leftrightarrow\quad
h>\frac{1}{\sqrt{2}},
\end{align*}
it follows from Theorem \ref{smetricinB_forconjugate} that
\begin{align*}
s_{\B^2}(x,y)&=h,\quad\text{if}\quad h>\frac{1}{\sqrt{2}}\\
s_{\B^2}(x,y)&=\frac{h\slash\sqrt{2}}{\sqrt{(1-h\slash\sqrt{2})^2+h^2\slash2}}
=\frac{h}{\sqrt{2h^2-2\sqrt{2}h+2}}\quad\text{otherwise,}
\\
w_{\B^2}(x,y)&=\frac{\sqrt{2}h}{|hi-(2-h)|}
=\frac{\sqrt{2}h}{\sqrt{2h^2-4h+4}}
=\frac{h}{\sqrt{h^2-2h+2}}.
\end{align*}
Consequently, if $h>1\slash\sqrt{2}$,
\begin{align}\label{val_1.04}
\frac{s_{\B^2}(x,y)}{w_{\B^2}(x,y)}=\sqrt{h^2-2h+2}
<\sqrt{(1\slash\sqrt{2})^2-2(1\slash\sqrt{2})+2}
=\sqrt{5\slash2-\sqrt{2}}
\approx1.04201
\end{align}
and, if $h\leq1\slash\sqrt{2}$ instead,
\begin{align}\label{quo_swl}
\frac{s_{\B^2}(x,y)}{w_{\B^2}(x,y)}=\sqrt{\frac{h^2-2h+2}{2h^2-2\sqrt{2}h+2}}.
\end{align}
Next, define a function $f:(0,1\slash\sqrt{2}]\to\R$,
\begin{align*}
f(h)=\frac{h^2-2h+2}{2h^2-2\sqrt{2}h+2}.    
\end{align*}
By differentiation,
\begin{align*}
f'(h)&=\frac{(2h-2)(2h^2-2\sqrt{2}h+2)-(4h-2\sqrt{2})(h^2-2h+2)}{(2h^2-2\sqrt{2}h+2)^2}\\
&=\frac{2((2-\sqrt{2})h^2-2h+2\sqrt{2}-2)}{(2h^2-2\sqrt{2}h+2)^2}.
\end{align*}
By the quadratic formula, $f'(h)=0$ holds when
\begin{align*}
h=\frac{2\pm\sqrt{4-4(2-\sqrt{2})(2\sqrt{2}-2)}}{2(2-\sqrt{2})}
=\frac{1\pm\sqrt{9-6\sqrt{2}}}{2-\sqrt{2}}.
\end{align*}
Here, the $\pm$-symbol must be minus, so that $0<h\leq1\slash\sqrt{2}$. Fix
\begin{align*}
h_0=\frac{1-\sqrt{9-6\sqrt{2}}}{2-\sqrt{2}}\approx0.48236.   
\end{align*}
Since $f'(0.1)>0$ and $f'(0.7)<1$, the function $f$ obtains its local maximum of the interval $(0,1\slash\sqrt{2}]$ at $h_0$. Thus, $\sqrt{f(h_0)}$ is the maximum value of the quotient \eqref{quo_swl} within the limitation $h\leq1\slash\sqrt{2}$. Since
\begin{align}\label{val_1.07}
\sqrt{f(h_0)}=\sqrt{\frac{h_0^2-2h_0+2}{2h_0^2-2\sqrt{2}h_0+2}}\approx1.07313    
\end{align}
is clearly greater than the upper value \eqref{val_1.04} for this same quotient in the case $h>1\slash\sqrt{2}$, the value \eqref{val_1.07} is the maximum value of the quotient of the triangular ratio metric and the quasi-metric $w_{\B^2}$ in the general case $0<h<1$. Thus, the lemma follows. 
\end{proof}

Even though the inequality of Lemma \ref{lem_sw_specialcase} is proven for a very specific choice of points $x,y\in\B^2$, the result itself is still relevant because several numerical tests suggest that it holds more generally.

\begin{conjecture}\label{con_ws}
For all $x,y\in\B^2$, the inequality $s_{\B^2}(x,y)\leq c\cdot w_{\B^2}(x,y)$ holds with the sharp constant
\begin{align*}
c=\sqrt{\frac{h_0^2-2h_0+2}{2h_0^2-2\sqrt{2}h_0+2}}\approx1.07313,\quad
h_0=\frac{1-\sqrt{9-6\sqrt{2}}}{2-\sqrt{2}}.
\end{align*}
\end{conjecture}

\begin{figure}
    \centering
    \scalebox{1}{\includegraphics{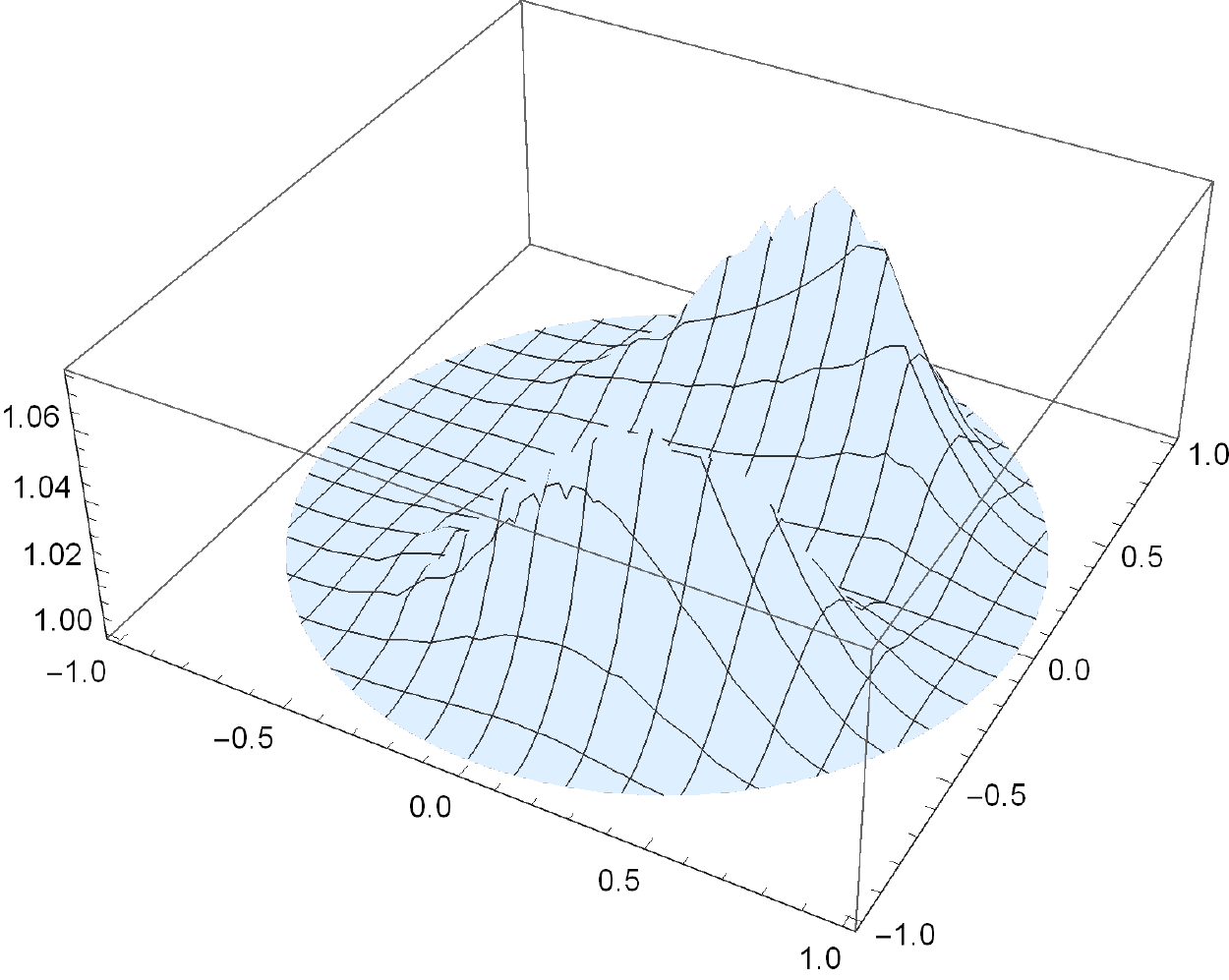}}
    \caption{Values of the quotient $s_{\B^2}(x,y)\slash w_{\B^2}(x,y)$ for different points $y\in\B^2$, when $x=0.6$ is fixed}
    \label{fig1}
\end{figure}

Figure \ref{fig1} also supports the assumption that the constant $c$ of Conjecture \ref{con_ws} is at most around 1.07. Consequently, the quasi-metric $w_{\B^2}$ is quite a good estimate for the triangular ratio metric in the unit disk. For instance, by choosing $c$ like above, we could use the value of $(c\slash2)\cdot w_{\B^2}(x,y)$ for the triangular ratio distance. Namely, if Conjecture \ref{con_ws} truly holds, our error with this estimate would be always less than $3.7$ percent.

\def\cprime{$'$} \def\cprime{$'$} \def\cprime{$'$}
\providecommand{\bysame}{\leavevmode\hbox to3em{\hrulefill}\thinspace}
\providecommand{\MR}{\relax\ifhmode\unskip\space\fi MR }
\providecommand{\MRhref}[2]{%
  \href{http://www.ams.org/mathscinet-getitem?mr=#1}{#2}
}
\providecommand{\href}[2]{#2}

\end{document}